\documentclass[a4paper, reqno, 11pt]{amsart}
\makeatletter

\@addtoreset{equation}{section}
\makeatother
\usepackage{setspace}
\usepackage{xcolor}
\usepackage{braket}
\usepackage[T1]{fontenc}
\everymath{\displaystyle}

\usepackage{amsmath}
\usepackage{amssymb}
\usepackage{latexsym}
\usepackage{amsthm}
\usepackage{textcomp}
\usepackage{hyperref}
\usepackage{mathtools}
\usepackage{here}
\newtheorem{Thm}{Theorem}[section]

\newtheorem{Lem}[Thm]{Lemma}
\newtheorem{Prop}[Thm]{Proposition}

\newtheorem{Cor}[Thm]{Corollary}

\theoremstyle{definition}

\newtheorem{Ass}[Thm]{Assumption}

\newtheorem{Rem}[Thm]{Remark}
\newtheorem{Exa}[Thm]{Example}

\begin{document}

\begin{abstract}
We establish the Moran--Hutchinson formula for attractors of finite systems of surjective similitudes on semimetric spaces. 
More precisely, for a complete, normal semimetric space satisfying strong regularity and geometric doubling, 
we prove that the open set condition implies that the Hausdorff measure of the attractor at the similarity dimension is positive and finite.  
We also prove the converse, specifically, positivity of the Hausdorff measure at the similarity dimension implies the open set condition. 
This provides a partial answer to a question posed by Bessenyei and P\'enzes in 2022. 
\end{abstract}

\title{The Moran--Hutchinson formula in semimetric spaces}
\author{Kazuki Okamura}
\date{\today}
\address{Department of Mathematics, Faculty of Science, Shizuoka University, 836, Ohya, Suruga-ku, Shizuoka, 422-8529, JAPAN.}
\email{okamura.kazuki@shizuoka.ac.jp}
\keywords{semimetric space; Hausdorff measure; attractor; open set condition}
\subjclass[2020]{54E25, 28A78, 47H09}
\maketitle

\section{Introduction}

Self-similar sets form a central class of fractal sets. 
Hutchinson \cite{Hutchinson1981} proved that every finite family of contractions on a complete metric space has a unique compact attractor. 
He also showed that, for contracting similitudes on a Euclidean space satisfying the open set condition, the Hausdorff dimension of the attractor equals its similarity dimension and the Hausdorff measure at that dimension is positive and finite. 
Based on Bandt and Graf \cite{BandtGraf1992}, 
Schief \cite{Schief1994} established a converse in the Euclidean setting,  
specifically, positivity of the Hausdorff measure at the similarity dimension implies the open set condition.

In arbitrary complete metric spaces, the open set condition alone no longer guarantees the dimension and measure conclusions which hold in the Euclidean case; see Schief \cite{Schief1996}.  
Under the additional assumption of doubling, Balogh and Rohner \cite{BaloghRohner2007} established the Moran--Hutchinson theorem and a corresponding converse.  
Rajala and Vilppolainen generalized this doubling-space theory to a broader class of iterated function systems in \cite[Theorem 4.9]{RajalaVilppolainen2011}.  
Their non-bijective similitude example in \cite[Example 4.5]{RajalaVilppolainen2011} also shows that the open set condition alone need not determine the Hausdorff dimension from the contraction ratios.  
In another direction, Wu and Yamaguchi \cite{WuYamaguchi2017} extended the doubling-space dimension theory to asymptotic self-similar systems.

Semimetric spaces provide a different extension of the classical setting. 
A semimetric retains symmetry and separation of points but need not satisfy the triangle inequality. 
Although a topology can be defined in terms of semimetric balls, a ball itself need not be open. 
Following the early work of Fr\'echet \cite{Frechet1906} and Wilson \cite{Wilson1931}, 
several additional conditions have been introduced to recover useful local and boundedness properties. 
Bessenyei and P\'ales \cite{BessenyeiPales2017} introduced a regularity condition which serves as a local substitute for the triangle inequality, 
while Bessenyei and P\'enzes \cite{BessenyeiPenzes2022} introduced normality. 
Kocsis and P\'ales \cite{KocsisPales2022}, and independently Bessenyei and P\'enzes \cite{BessenyeiPenzes2022}, established the existence and uniqueness of attractors for finite families of weak contractions on complete regular normal semimetric spaces.

In this paper, we give a partial answer to the dimension problem posed by Bessenyei and P\'enzes. 
We consider a finite family of surjective contracting similitudes on a complete normal semimetric space whose basic triangle function satisfies a strong regularity condition and which is geometrically doubling. 
Under the open set condition, we prove that the Hausdorff measure at that dimension is positive and finite, and that the Hausdorff and Minkowski dimensions of the attractor equal its similarity dimension. 
Conversely, under the same space and mapping assumptions, positivity of this Hausdorff measure implies the open set condition.  
Rajala and Vilppolainen \cite{RajalaVilppolainen2011} consider a wider class of maps, while we consider only surjective similitudes. 
However, we work with semimetric spaces instead of metric spaces. 
Thus, our results extend the theory in a different direction and do not directly generalize the results of \cite{BaloghRohner2007} or \cite{RajalaVilppolainen2011}. 

This paper is organized as follows. 
Section \ref{sec:frame} introduces the framework and notation. 
In Section \ref{sec:inv-measure}, we construct an invariant probability measure (Proposition \ref{prop:exist-inv-measure}). 
Section \ref{sec:MH-formula} contains the proof of the Moran--Hutchinson formula (Theorem \ref{thm:main}). 
In Section \ref{sec:osc}, we prove the necessity of the open set condition (Theorem \ref{thm:osc-necessary}). 
Finally, Section \ref{sec:exa} gives examples.

\section{Framework}\label{sec:frame}

Let $X$ be a nonempty set. 
We say that a map $d : X \times X \to [0, \infty)$ is a {\it semimetric} on $X$ if $d(x,y) = 0$ if and only if $x=y$, and $d(x,y) = d(y,x)$ for every $x, y \in X$.  
For $x \in X$ and $r > 0$, define the ball $B(x,r) \coloneqq \left\{y \in X | d(x,y) < r \right\}$. 
We say that a subset $U$ of $X$ is an open set if 
for every $x \in U$, there exists $r > 0$ such that $B(x,r) \subset U$. 
This defines a topology on $X$. 
We remark that a ball may not be an open set.  

A subset of $X$ is closed if its complement is open. 
For $H \subset X$, we denote the closure of $H$ by $\overline{H}$. 
We say that a sequence $(x_n)_n$ in $X$ is convergent if  $\lim_{n \to \infty} d(x_n, x) = 0$ for some $x \in X$, and it is a Cauchy sequence if $\lim_{n, m \to \infty} d(x_n, x_m) = 0$. 
We say that $(X,d)$ is complete if every Cauchy sequence converges, and sequentially compact if every sequence has a convergent subsequence.

We introduce a notion which plays the role of the classical triangle inequality. 
We say that a map $\Phi \colon [0,\infty) \times [0,\infty) \to [0,\infty]$ is a {\it triangle function} if $\Phi(0,0) = 0$, $\Phi$ is symmetric, that is, $\Phi(s,t) = \Phi(t,s)$ for every $s, t \in [0,\infty)$, $\Phi$ is nondecreasing  in each of its arguments, and furthermore 
\[ d(x,z) \le \Phi(d(x,y), d(y,z)), \quad x, y, z \in X.  \]
For $u, v \in [0,\infty)$, let 
\[ \mathcal{I}(u,v) \coloneqq \left\{(x,y) \in X^2 \colon  \textup{ there exists }   z \in X \textup{ such that } d(x,z) \le u, d(y,z) \le v\right\}, \]
and define  
\[ \Phi_d (u,v) \coloneqq \sup\left\{d(x,y) \colon (x,y) \in \mathcal{I}(u,v) \right\} \in [0,\infty], \]
where we let $\sup \emptyset \coloneqq 0$. 
For every $x \in X$, $(x,x) \in \mathcal{I}(u,v)$, and in particular, $\mathcal{I}(u,v) \ne \emptyset$. 
We call $\Phi_d$ the {\it basic triangle function}. 
It is a triangle function and $\Phi_d (u,v) \le \Phi(u,v)$. 

We say that a semimetric space $(X,d)$ satisfies the {\it regular} property for a triangle function $\Phi$ if $\Phi$ is continuous at $(0,0)$. 
We say that $(X,d)$ satisfies the  {\it normal} property if $\Phi (u,v) \ne \infty$ for $u, v \in [0,\infty)$. 
Every metric space is a semimetric space satisfying the regular and normal properties for the basic triangle function. 

For $A \subset X$, 
\[ \textup{int}(A) \coloneqq \left\{x \in A \colon \textup{ there exists $\epsilon > 0$ such that $B(x,\epsilon) \subset A$} \right\}. \]

\begin{Lem}\label{lem:interior-open}
Assume that a semimetric space $(X,d)$ is regular. Then,\\
(i) For every $A \subset X$, $\textup{int}(A)$ is an open subset of $X$. \\
(ii) For every $r > 0$, there exists $\delta > 0$ such that $B(x,\delta) \subset \textup{int}(B(x,r))$ for every $x \in X$. 
\end{Lem}

Part (ii) is the same as \cite[Lemma 2]{KocsisPales2022}. 

\begin{proof}
(i) Let $x \in \textup{int}(A)$. 
Then there exists $\epsilon > 0$ such that $B(x,\epsilon) \subset A$. 
By the regularity, there exists $\delta > 0$ such that $\Phi_d (\delta, \delta) < \epsilon$. 
It suffices to show that $B(x,\delta) \subset  \textup{int}(A)$.  
Let $y \in B(x,\delta)$. 
Then, for every $z \in B(y,\delta)$, 
\[ d(x,z) \le \Phi_d (d(x,y), d(y,z)) \le \Phi_d (\delta, \delta) < \epsilon. \]
Hence, $z \in B(x,\epsilon) \subset A$. 
Hence $B(y,\delta) \subset A$, which means $y \in \textup{int}(A)$. 

Part (ii) follows by the same argument. 
\end{proof}

By Lemma \ref{lem:interior-open}, we can show that every regular semimetric space is Hausdorff. 

The following result means that the topological and sequential closures of a set coincide. 

\begin{Lem}\label{lem:seq-closure}
If a semimetric space $(X,d)$ is regular, then for every $H \subset X$, 
\begin{equation}\label{eq:closure-limit} 
\overline{H} = \left\{x \in X \colon \textup{there exists a sequence $(x_n)_n$ in $H$ such that $\lim_{n \to \infty} d(x_n, x) = 0$} \right\}. 
\end{equation}
\end{Lem}

\begin{proof}
Let $\widetilde{H}$ be the set on the right-hand side of \eqref{eq:closure-limit}. 
$\widetilde{H}$ contains $H$. 
We show that $\widetilde{H}$ is closed. 
This is equivalent to saying that $X \setminus \widetilde{H}$ is open. 
Let $x \in X \setminus \widetilde{H}$. 
Assume that for every $n \ge 1$, there exists $x_n \in \widetilde{H} \cap B(x, 2^{-n})$. 
Then, for every $n \ge 1$, there exists $y_n \in H$ such that $d(x_n, y_{n}) \le 2^{-n}$. 
By the regularity, 
\[ d(y_{n}, x) \le \Phi_d \left(d(y_{n}, x_n), d(x_n, x)\right) \le \Phi_d (2^{-n}, 2^{-n}) \to 0, \quad n \to \infty. \]
Hence $x \in \widetilde{H}$. 
This is a contradiction. 
Hence there exists $n \ge 1$ such that $ \widetilde{H} \cap B(x, 2^{-n}) = \emptyset$. 
Thus we see that $X \setminus \widetilde{H}$ is open and that $\overline{H} \subset \widetilde{H}$. 

Let $x \in \widetilde{H}$. 
Let $U$ be an open subset of $X$ containing $x$. 
Then there exists $r > 0$ such that $B(x,r) \subset U$ and there exists $y \in H$ such that $y \in B(x,r)$. 
Hence $H \cap U \ne \emptyset$. 
Hence $x \in \overline{H}$. 
Thus we see that $\widetilde{H} \subset \overline{H}$. 
\end{proof}

By this lemma, a subset of a complete regular semimetric space is complete if and only if it is closed. 
A regular semimetric space  $(X,d)$ is compact if and only if it is sequentially compact.

We now state some consequences of the normal property. 
We say that a subset $H$ of $X$ is {\it bounded} if there exist $x_0 \in X$ and $r_0 > 0$ such that $H \subset B(x_0, r_0)$. 
Let 
\[ \textup{diam}(H) \coloneqq \sup\left\{d(x,y) | x, y \in H \right\} \] 
for $H \subset X$, $H \ne \emptyset$, and we call it the {\it diameter} of $H$. 

\begin{Lem}[{\cite[Proposition 3]{BessenyeiPenzes2022}}]\label{lem:bounded-basic}
Let $(X,d)$ be a normal semimetric space and $H$ be a nonempty subset of $X$. 
Then, $H$ is bounded if and only if $\textup{diam}(H) < \infty$. 
\end{Lem}

We say that a subset $H$ of $X$ is {\it totally bounded} if for every $\epsilon > 0$ there exist finitely many points $x_1, \dots, x_n \in H$ such that $H \subset \bigcup_{i=1}^{n} B(x_i, \epsilon)$. 
By \cite[Lemma 2]{BessenyeiPenzes2022}, 
we see that if $(X,d)$ is complete and regular and $H \subset X$, then $H$ is compact if and only if $H$ is closed and totally bounded. 

{\it Hereafter, we assume that $(X,d)$ is a complete semimetric space satisfying the regular and normal properties for the basic triangle function.}

We adopt Matkowski-Rus's definition of weak contractions \cite{Matkowski1975,Rus2001}.  
We say that a map $\phi : [0,\infty) \to [0,\infty)$ is a {\it comparison function} if 
$\phi$ is nondecreasing, and 
$\lim_{n \to \infty} \phi^n (t) = 0$ for every $t > 0$. 
Let $\phi$ be a comparison function. 
We say that a map $T : X \to X$ is a {\it $\phi$-contraction} if 
$d(Tx, Ty) \le \phi(d(x,y))$ for every $x, y \in X$.

We have the following. 
\begin{Lem}\label{lem:comparison-basic}
(i) If $\phi$ is a comparison function, then, $\phi(0) = 0$ and  $\phi(t) < t$ for every $t > 0$. \\
(ii) If $\phi_1$ and $\phi_2$ are right-continuous comparison functions and $\phi \coloneqq \max\{\phi_1, \phi_2\}$, then, $\phi$ is also a right-continuous comparison function. 
\end{Lem}

We now recall some consequences of the regular property. 

\begin{Thm}[Bessenyei and P\'ales {\cite[Theorem 1]{BessenyeiPales2017}} ]\label{thm:Bessenyei-Pales-FP} 
Let $\phi$ be a comparison function and let $T : X \to X$ be a $\phi$-contraction. 
Then, there exists a unique fixed point of $T$. 
\end{Thm}

Let $\mathcal{K}(X)$ be the set of nonempty compact subsets of $X$. 

\begin{Thm}[Kocsis and P\'ales {\cite[Theorem 17]{KocsisPales2022}}; Bessenyei and P\'{e}nzes {\cite[Theorem 2]{BessenyeiPenzes2022}}]\label{thm:attractor}
Let $\ell \ge 1$. 
For each $i \in \{1, \dots, \ell\}$, 
let $f_i$ be a $\phi_i$-contraction on $X$ where $\phi_i$ is a right-continuous comparison function.
Then, there exists a unique $K \in \mathcal{K}(X)$ such that $K = \bigcup_{i=1}^{\ell} f_i (K)$. 
\end{Thm}

We call the set $K$ in the above theorem the {\it attractor} of $\{f_i\}_i$.

\begin{Rem}
A semimetric space $(X,d)$ is called a {\it quasi-metric space} if there exists $C < \infty$ such that $d(x,z) \le C(d(x,y) + d(y,z))$ for every $x, y, z \in X$. 
This is also known as a {\it $b$-metric space} \cite{CJT2018}. 
Every quasi-metric space is normal and regular. 
Properties of quasi-metric spaces were investigated by Bourbaki \cite{Bourbaki2007}. 
\end{Rem}

We follow \cite{BessenyeiPenzes2022} for the definitions of Hausdorff measure and Hausdorff dimension. 
These are defined in the same manner as in the case of metric spaces. 

For $A \subset X$, $\alpha \ge 0$ and $\delta > 0$, let 
\[ \mathcal{H}^{\alpha}_{\delta}(A) \coloneqq \inf\left\{ \sum_{i=1}^{\infty} \textup{diam}(B_i)^{\alpha} \colon A \subset \bigcup_{i} B_{i}, \textup{diam}(B_i) < \delta   \right\} \]
and $\mathcal{H}^{\alpha}(A) \coloneqq \lim_{\delta \to +0} \mathcal{H}^{\alpha}_{\delta}(A)$. 
The Hausdorff dimension of $A \subset X$ is defined by 
\[ \dim_H (A) \coloneqq \sup\{r > 0 \colon \mathcal{H}^r (A) = \infty \}. \]
We see that $\mathcal{H}^t (A) = 0$ for every $t > r$ whenever $\mathcal{H}^r (A) < \infty$ and that $\dim_H (A) = \inf\{r > 0 \colon \mathcal{H}^{r}(A) = 0\}$. 

We define the Minkowski dimension in the same manner as in the case of metric spaces. 
For a totally bounded subset $A$ of $X$, 
let $N(A, \epsilon)$ be the minimum number of sets of diameter less than $\epsilon$ needed to cover $A$. 
Define the upper and lower Minkowski dimensions by 
\[ \overline{\dim}_M A \coloneqq \limsup_{\epsilon \to +0} \frac{\log N(A, \epsilon)}{\log 1/\epsilon},  \textup{ and } \underline{\dim}_M A \coloneqq \liminf_{\epsilon \to +0} \frac{\log N(A, \epsilon)}{\log 1/\epsilon},  \]
respectively. 
If these two values agree, then we call their common value the {\it Minkowski dimension} of $A$ and denote it by $\dim_M A$. 
Then, 
\[ \dim_H A \le \underline{\dim}_M A \le \overline{\dim}_M  A.  \]

\section{Existence of an invariant measure}\label{sec:inv-measure}

The goal of this section is to show the following. 

\begin{Prop}\label{prop:exist-inv-measure}
Assume that $p_1, \dots, p_{\ell} > 0$, $\sum_{j = 1}^{\ell} p_j = 1$. 
For each $j \in \{1, \dots, \ell\}$, 
let $f_j$ be a $\phi_j$-contraction on $X$ where $\phi_j$ is a right-continuous comparison function. 
Let $K$ be the attractor of $\{f_j\}_j$. 
Then, there exists a Borel probability measure $\mu$ supported on $K$ such that 
\begin{equation}\label{eq:invariance-measure} 
\mu = \sum_{j = 1}^{\ell} p_j \mu \circ f_j^{-1}. 
\end{equation}
\end{Prop}

For $\xi = (\xi_i)_i \in \{1, \dots, \ell\}^{\mathbb N}$, 
let  $f_{\xi|_{n}} \coloneqq f_{\xi_1} \circ \cdots \circ f_{\xi_n}$. 

\begin{Lem}
The set $\bigcap_{n \ge 1} f_{\xi|_{n}}(K)$ is a singleton contained in $K$. 
\end{Lem}

\begin{proof}
Since each comparison function $\phi_j$ is right-continuous, each $f_j$ is continuous. 
Hence $f_{\xi|_{n}}$ is continuous. 
Since $K$ is compact, $f_{\xi|_{n}}(K)$ is also compact. 
Since $f_j (K) \subset K$ for each $j$, the sets $f_{\xi|_{n}}(K)$ form a decreasing sequence. 
Hence, $\bigcap_{n \ge 1}  f_{\xi|_{n}}(K)$ is not empty. 

Let $\phi \coloneqq \max_{j \in \{1, \dots, \ell\}} \phi_j$. 
Then, by Lemma \ref{lem:comparison-basic}, this is also a right-continuous comparison function. 
We see that by induction on $n$, 
$\textup{diam}( f_{\xi|_{n}}(K)) \le \phi^n (\textup{diam}(K))$. 
Since $K$ is compact, by Lemma \ref{lem:bounded-basic}, $\textup{diam}(K) < \infty$.  
 
Let $x_1, x_2 \in \bigcap_{n \ge 1}  f_{\xi|_{n}}(K)$. 
Then, $d(x_1, x_2) \le \textup{diam}( f_{\xi|_{n}}(K))$ for each $n$. 
Hence, $d(x_1, x_2) = 0$ and then $x_1 = x_2$. 
\end{proof}

For ease of notation, let $\Omega \coloneqq \{1, \dots, \ell\}^{\mathbb N}$. 
Define a map $\Psi \colon \Omega \to K$ by $\{\Psi(\xi)\} = \bigcap_{n \ge 1} f_{\xi|_{n}}(K)$.  
We equip $\Omega$ with the product topology. 

\begin{Lem}\label{lem:conti-surj}
The map $\Psi$ is a continuous surjection. 
\end{Lem}

\begin{proof}
First we show that $\Psi$ is continuous. 
For $\xi = (\xi_i)_{i \ge 1}, \eta = (\eta_i)_{i \ge 1} \in \Omega$, 
let $|\xi \wedge \eta| \coloneqq \min\{i \ge 1 \colon \xi_i \ne \eta_i\}$ and $d_{\Omega} (\xi, \eta) \coloneqq \exp(-|\xi \wedge \eta|)$. 
We remark that $|\xi \wedge \eta| = \infty$ if and only if $\xi = \eta$. 
We assume that $\min \emptyset = \infty$ and $\exp(-\infty) = 0$. 
This defines a metric on $\Omega$ and the metric topology is identical to the product topology. 
For $\xi \ne \eta$, we see that 
\[ d(\Psi(\xi), \Psi(\eta)) \le \textup{diam}\left(f_{\xi|_{n-1}} (K) \right) \le \phi^{n -1}(\textup{diam}(K)) \]
\[ = \phi^{-\log d_{\Omega}(\xi, \eta) -1}(\textup{diam}(K)), \quad n = |\xi \wedge \eta|,  \]
where we let $\xi|_0 \coloneqq \emptyset$ and recall that $f_{\emptyset}$ is the identity map on $X$. 
Since $\phi$ is a comparison function, 
for every $\epsilon > 0$, there exists $\delta > 0$ such that $d_{\Omega}(\xi, \eta) < \delta$ implies $d(\Psi(\xi), \Psi(\eta)) < \epsilon$. 
Hence $\Psi$ is continuous. 

We now show that $\Psi$ is surjective. 
Let $x \in K$. 
We define a sequence $\xi = (\xi_i)_i \in \Omega$ inductively. 
Since $K$ is the attractor, there exists $j  \in \{1, \dots, \ell\}$ such that $x \in f_{j} (K)$, and we let $\xi_1 \coloneqq j$ by choosing one such $j$. 
We choose $x_1 \in K$ such that $x = f_{\xi_1} (x_1)$. 
In the same manner, there exists $j \in \{1, \dots, \ell\}$ such that $x_1 \in f_{j} (K)$, and we let $\xi_2 \coloneqq j$ by choosing one such $j$. 
We choose $x_2 \in K$ such that $x_1 = f_{\xi_2} (x_2)$. 
We repeat this procedure and obtain a sequence $\xi = (\xi_i)_i \in \Omega$. 
 Then, for every $n \ge 1$, $x \in f_{\xi|_{n}} (K)$ and hence $x = \Psi(\xi)$. 
\end{proof}

Let $\iota_{\bf p}$ be the probability measure on $\{1,\dots, \ell\}$ such that $\iota_{\bf p} (\{j\}) = p_j$ for each $j$. 
Let $\nu_{\bf p}$ be the product measure $\iota_{\bf p}^{\otimes \mathbb N}$ on $\Omega$. 
By Lemma \ref{lem:conti-surj}, $\Psi$ is a Borel measurable map. 
Let $\mu_{\bf p}$ be the push-forward measure of $\nu_{\bf p}$ by the map $\Psi$.  
This is a Borel probability measure on $K$. 
We regard it as a probability measure on $X$, extended by zero outside $K$. 

We show that $\mu_{\bf p} = \sum_{j = 1}^{\ell} p_j \mu_{\bf p} \circ f_{j}^{-1}$. 
Let $A$ be a Borel measurable subset of $K$. 
Let $\tilde{\xi} \coloneqq (\xi_{i+1})_{i \ge 1}$ for $\xi = (\xi_i)_{i \ge 1}$.  
Then, $\Psi(\xi) = f_{\xi_1}\left(\Psi (\tilde{\xi}) \right)$.  
Therefore, 
\[ \mu_{\bf p} (A) = \nu_{\bf p}  (\Psi^{-1}(A)) = \sum_{j = 1}^{\ell} \nu_{\bf p}  \left(  \left\{\xi \in \{1, \dots, \ell\}^{\mathbb N}  \colon \xi_1 = j, \, \Psi(\xi) \in A \right\} \right)  \] 
\[ = \sum_{j = 1}^{\ell} \nu_{\bf p}  \left( \left\{\xi  \in \{1, \dots, \ell\}^{\mathbb N} \colon \xi_1 = j, \, \tilde{\xi} \in \Psi^{-1}(f_j^{-1}(A)) \right\}  \right) \]
\[ = \sum_{j = 1}^{\ell} p_j \nu_{\bf p}  \left(\Psi^{-1}(f_j^{-1}(A))  \right) =  \sum_{i=1}^{\ell} p_i \mu_{\bf p} (f_i^{-1} (A)).  \]
Since $f_j (K) \subset K$ for each $j$, $\mu_{\bf p} \circ f_j^{-1}$ vanishes on $K^c$. 
Hence \eqref{eq:invariance-measure} holds for every Borel measurable subset of $X$. 

\begin{Rem}
We have {\it not} shown the uniqueness of the invariant measure satisfying \eqref{eq:invariance-measure}. 
\end{Rem}

\section{Moran--Hutchinson formula}\label{sec:MH-formula}

In this section we establish the Moran--Hutchinson formula for the attractor of a finite family of similitudes. 
As an outline, we follow the strategy in the proof of  \cite[Theorem 2.2.2]{BishopPeres2017}. 
However, we need to modify several parts of the proof.

\begin{Ass}\label{ass:two-additional-metric}
(i) (strong regularity)
\[ \limsup_{\epsilon \to +0} \frac{\Phi_d (\epsilon, \epsilon)}{\epsilon} < \infty.  \]
(ii) (geometric doubling)
There exists $N_0 \in \mathbb{N}$ such that for every $x \in X$ and every $r > 0$, there exist $x_1, \dots, x_{N_0} \in X$ such that $B(x,r) \subset \bigcup_{i=1}^{N_0} B(x_i, r/2)$. 
\end{Ass}

Every quasi-metric space satisfies the strong regularity above. 
The following corresponds to \cite[Lemma 3.3]{BaloghRohner2007}. 

\begin{Lem}\label{lem:number-disjoint}
Suppose Assumption \ref{ass:two-additional-metric} holds. 
Let $b > a > 0$. 
Let $r > 0$. 
Let $W_1, \dots, W_N$ be disjoint subsets of $X$ such that for each $i$, $W_i$ contains a ball of radius $ar$ and is contained in a ball of radius $br$. 
Assume that there exists a subset $A$ of $X$ such that  
$\textup{diam}(A) \le r$ and $A \cap W_i \ne \emptyset$ for each $i \in \{1, \dots, N\}$. 
Then, there exist  $n_0 (a,b) \in \mathbb{N}$ and $R_0 (a,b) > 0$ such that whenever $r \in (0, R_0 (a,b))$ and the above conditions hold, $N \le N_0^{n_0 (a,b)}$. 
\end{Lem}

\begin{proof}
By the assumption, for each $i$, there exist two points $y_i, z_i \in X$ such that $B(y_i, ar) \subset W_i \subset B(z_i, br)$. 

Step 1. We show that there exist $x \in A$  and  two constants $C(b) >  1$ and $R_1 (b) > 0$ such that $W_i \subset B(x, C(b) r)$ for each $i$ and every $r \in (0, R_0 (b))$. 

Take a point $x \in A$. 
Let $x_i \in A \cap W_i$.   
Then, for every $w \in W_i$, 
\[ d(x,w) \le \Phi_d (d(x,x_i), d(x_i,w)) \le \Phi_d \left(r, \Phi_d (br, br)\right)\]
\[  \le  \Phi_d \left(u,  u\right), \quad u = \max\left\{r, \Phi_d (br,br)\right\}. \]
By Assumption \ref{ass:two-additional-metric} (i), 
there exist $R_1 > 0$ and $C_1 > 1$ such that for every $r \in (0, R_1)$, $\Phi_d (r,r) \le C_1 r$. 
The claim follows for $R_1 (b) \coloneqq R_1/(1+ C_1 b)$ and $C(b) \coloneqq C_1 (1+ C_1 b) $. 

Step 2. For every $r \in (0, R_0 (b))$, $B(y_i, ar), 1 \le i \le N$, are disjoint balls  contained in $B(x, C(b) r)$. 
Then there exist $n_0 = n_0 (a,b) \in \mathbb{N}$ and $R_2 (a,b) > 0$ such that for every $r \in (0, R_2 (a,b))$,  $\Phi_d (C(b) r/2^{n_0}, C(b) r/2^{n_0}) < ar$. 

Let $R_0 (a,b) \coloneqq \min\{R_1 (b), R_2 (a,b) \}$. 
Fix $r \in \left(0, R_0 (a,b)\right)$. 
Then by Assumption \ref{ass:two-additional-metric} (ii), 
there exist $v_1, \dots, v_{N_0^{n_0}} \in X$ such that $B(x,C(b)r) \subset \bigcup_{k=1}^{N_0^{n_0}} B(v_k, C(b) r/2^{n_0})$. 
If $N > N_0^{n_0}$, then by the pigeonhole principle, there exist $i \ne j$ and $k$ such that $y_i \in B(v_k, C(b) r/2^{n_0})$ and $y_j \in B(v_k, C(b) r/2^{n_0})$. 
It holds that 
\[  d(y_i, y_j) \le \Phi_d (d(y_i, v_k), d(v_k, y_j)) \le \Phi_d (C(b) r/2^{n_0}, C(b) r/2^{n_0}) < ar. \]
This contradicts the assumption that $B(y_i, ar), 1 \le i \le N$, are disjoint. 
Hence $N \le N_0^{n_0}$. 
\end{proof}

\begin{Rem}
In the above proof, the geometric doubling condition in Assumption \ref{ass:two-additional-metric} (ii) can be weakened to a {\it local} geometric doubling condition: 
there exists $N_0 \in \mathbb{N}$ and $r_0 > 0$ such that for every $x \in X$ and every $r \in (0, r_0)$, there exist $x_1, \dots, x_{N_0} \in X$ such that $B(x,r) \subset \bigcup_{i=1}^{N_0} B(x_i, r/2)$. 
\end{Rem}

\begin{Ass}\label{ass:additional-f}
(i) (open set condition)  There exists a nonempty, bounded open subset $V$ of $X$ such that $\bigcup_j f_j (V) \subset V$ and 
$f_i (V) \cap f_j (V) = \emptyset$ for $i \ne j$. \\
(ii) (surjectivity and similarity) For each $j$, $f_j$ is a surjective map and a similitude with contraction ratio $r_j$, specifically, 
there exists $r_j \in (0,1)$ such that $d(f_j (x_1), f_j (x_2)) = r_j d(x_1, x_2)$ for every $x_1, x_2 \in X$. 
\end{Ass}

If $(X,d)$ is a Euclidean space and $f$ is a similitude on $X = \mathbb{R}^d$ such that $d(f(x_1), f(x_2)) = r d(x_1, x_2)$ for every $x_1, x_2 \in X$, 
then there exists a $d \times d$ orthogonal matrix $O$ and $x_0 \in X$ such that $f(x) = rOx + x_0$ for every $x \in X$. 
In particular $f$ is bijective. 

\begin{Lem}\label{lem:K-included}
Under Assumption \ref{ass:additional-f} (i), $K \subset \overline{V}$. 
\end{Lem}

\begin{proof}
Since $f_j$ is continuous, $f_j (\overline{V}) \subset \overline{f_j (V)}$. 
Hence  
\[ \bigcup_j f_j (\overline{V}) \subset \bigcup_j \overline{f_j (V)} \subset \overline{V}.\]  
Since $V$ is not empty, we can take a point $v_0 \in V$. 
Since $K$ is compact, by recalling that $(X,d)$ is normal, $M \coloneqq \sup_{x \in K} d(x, v_0) < \infty$. 
Let $x \in K$. 
Then  there exists $\xi \in \{1,\dots, \ell\}^{\mathbb N}$ such that $\Psi(\xi) = x$. 
For every $n$, $x \in f_{\xi|_n} (K)$. 
Hence, $d(x,  f_{\xi|_n}(v_0)) \le \phi^n (M)$. 
Since $ f_{\xi|_n}(v_0) \in V$ for every $n$, we see that $x \in \overline{V}$. 
By \eqref{eq:closure-limit}, we see that $K \subset \overline{V}$. 
\end{proof}

Let $r_{\min} \coloneqq \min_j r_j \in  (0,1), r_{\max} \coloneqq \max_j r_j \in (0,1)$. 
Let $\{1,\dots, \ell\}^{*} \coloneqq \bigcup_{n \ge 0} \{1,\dots, \ell\}^{n}$, where we let $\{1,\dots,\ell\}^{0} \coloneqq \{\emptyset\}$. 
We let $r_{\emptyset} \coloneqq 1$ and $f_{\emptyset}$ be the identity map of $X$. 
For $\sigma = (\sigma_i)_{i=1}^{m} \in  \{1,\dots, \ell\}^{*}$, 
let $|\sigma| \coloneqq m$, $r_{\sigma} \coloneqq r_{\sigma_1} \cdots r_{\sigma_m}$, $p_{\sigma} \coloneqq p_{\sigma_1} \cdots p_{\sigma_m}$ and $f_{\sigma} \coloneqq f_{\sigma_1} \circ \cdots \circ f_{\sigma_m}$. 

\begin{Lem}\label{lem:ball-surjection}
Suppose Assumption \ref{ass:additional-f} (ii) holds. 
Then, for every $\sigma \in \{1,\dots, \ell\}^*$, every $\delta > 0$ and every $x \in X$, 
$B(f_{\sigma}(x), \delta r_{\sigma}) = f_{\sigma}(B(x,\delta))$. 
\end{Lem}

By this lemma, $f_j$ is an open map and hence $f_j$ is a homeomorphism of $X$ onto itself. 

\begin{proof}
Let $y \in B(f_{\sigma}(x), \delta r_{\sigma})$. 
Since $f_{\sigma}$ is surjective, 
there exists $z$ such that $y = f_{\sigma}(z)$. 
We see that $\delta r_{\sigma} > d(f_{\sigma}(x), f_{\sigma}(z)) = r_{\sigma} d(x,z)$.
Hence $d(x,z) < \delta$. 
Hence $B(f_{\sigma}(x), \delta r_{\sigma}) \subset f_{\sigma}(B(x,\delta))$. 
Let $z \in B(x,\delta)$. 
Then $d(f_{\sigma}(x), f_{\sigma}(z)) = r_{\sigma} d(x,z) < r_{\sigma} \delta$. 
\end{proof}

We say that $\Pi \subset \{1,\dots, \ell\}^{*}$ is a {\it cut-set} if for every $\xi \in \{1,\dots, \ell\}^{\mathbb N}$ there exists $\sigma \in \Pi$ such that $\sigma$ is a prefix of $\xi$. 
We say that a cut-set is {\it minimal} if there exists no element in the cut-set which is a prefix of another element. 
Then, by \eqref{eq:invariance-measure}, 
we obtain that for every minimal cut-set $\Pi$, 
\begin{equation}\label{eq:invariance-min-cutset}
\mu_{\bf p} = \sum_{\sigma \in \Pi} p_{\sigma} \mu_{\bf p} \circ f_{\sigma}^{-1}.
\end{equation}
See \cite[Lemma 2.2.4 (ii)]{BishopPeres2017}. 

For $\rho \in (0, r_{\min})$, 
let 
\begin{equation}\label{eq:def-cutset-1} 
\Pi_{\rho} \coloneqq \left\{\sigma \in \{1,\dots, \ell\}^{*} \colon r_{\sigma} \le \rho < r_{\sigma^{\prime}} \right\}, 
\end{equation}
where $\sigma^{\prime} \coloneqq (\sigma_1, \dots, \sigma_{|\sigma|-1})$. 
Since $r_{\max} < 1$, for every $\xi \in \{1,\dots, \ell\}^{\mathbb N}$, 
$\lim_{n \to \infty} r_{\xi|_n} = 0$. 
Hence there exists a unique $n$ such that $r_{\xi|_n} \le \rho < r_{\xi|_{n-1}}$, and then $\xi|_n \in \Pi_{\rho}$. 
Let $\sigma, \tau \in \Pi_{\rho}$, and suppose that $\sigma$ is a prefix of $\tau$. 
Then $|\sigma| \le |\tau|$. 
If $|\sigma| < |\tau|$, then $r_{\tau^{\prime}} \le r_{\sigma}$. 
This cannot occur. 
Hence, $|\sigma| = |\tau|$ and then $\sigma = \tau$. 
Thus we see that $\Pi_{\rho}$ is a minimal cut-set.

\begin{Prop}\label{prop:mass-dist}
Suppose that Assumptions \ref{ass:two-additional-metric} and \ref{ass:additional-f}  hold. 
Let $\alpha$ be a positive real number such that $\sum_j r_j^{\alpha} = 1$. 
Let $p_j \coloneqq r_j^{\alpha}$ for each $j$ and ${\bf p} \coloneqq (p_j)_j$. 
Then there exist $C_0 \in (0,\infty)$ and $\delta_0 > 0$ such that 
$\mu_{\bf p}(U) \le C_0 \rho^{\alpha}$ for every $\rho \in (0, \delta_0)$ and every open set $U$ with $\textup{diam}(U) \le \rho$. 
\end{Prop}

\begin{proof}
Let $b \coloneqq 1 + \textup{diam}(V)$. 
Let $\rho \in \left(0, \min\{R_0 (b), r_{\min} \}\right)$. 
By \eqref{eq:invariance-min-cutset}, we obtain that 
\[ \mu_{\bf p}(U) =  \sum_{\sigma \in \Pi_{\rho}} r_{\sigma}^{\alpha} \mu_{\bf p} \left( f_{\sigma}^{-1} (U) \right) \le \rho^{\alpha} \left|\left\{\sigma \in \Pi_{\rho} \colon \mu_{\bf p} \left( f_{\sigma}^{-1} (U) \right) > 0 \right\}\right|.  \]
If $\mu_{\bf p} \left( f_{\sigma}^{-1} (U) \right) > 0$, then $f_{\sigma}^{-1} (U) \cap K \ne \emptyset$. 
By Lemma \ref{lem:K-included}, $f_{\sigma}^{-1} (U) \cap \overline{V} \ne \emptyset$. 
Since $U$ is open and $f_{\sigma}$ is continuous, 
$f_{\sigma}^{-1} (U) \cap V \ne \emptyset$, and hence $U \cap f_{\sigma}(V) \ne \emptyset$. 

Since $\Pi_{\rho}$ is a minimal cut-set, by the open set condition, 
$\{ f_{\sigma}(V) \}_{\sigma \in \Pi_{\rho}}$ are disjoint. 
Since $V$ is a nonempty bounded open set, 
there exist $z \in V$ and $\delta > 0$ such that 
$B(z,\delta) \subset V \subset B\left(z, 1+\textup{diam}(V) \right)$. 
We remark that by the normal property, $\textup{diam}(V) < \infty$.  
By Lemma \ref{lem:ball-surjection}, 
\[ B\left(f_{\sigma}(z), \rho \delta r_{\min}\right) \subset B(f_{\sigma}(z), \delta r_{\sigma}) \subset f_{\sigma}(V) \subset B\left(f_{\sigma}(z), \rho(1+ \textup{diam}(V))\right).\] 
 We apply Lemma \ref{lem:number-disjoint} to the case where $a = \delta r_{\min}$ and $b = 1+\textup{diam}(V)$ and $r = \rho$. 
 Then, $|\{\sigma \in \Pi_{\rho} \colon U \cap f_{\sigma}(V) \ne \emptyset\}| \le N_0^{n_0 (a,b)}$. 
 Now the assertion holds for $C_0 =  N_0^{n_0 (a,b)}$ and $\delta_0 = \min\{R_0 (b), r_{\min} \}$. 
\end{proof}

\begin{Cor}\label{cor:non-atomic}
Under the assumptions of Proposition \ref{prop:mass-dist}, $\mu_{\bf p}$ is non-atomic.
\end{Cor}

\begin{proof}
Let $x \in X$. 
By Assumption \ref{ass:two-additional-metric} (i), 
there exist $C_1 > 0$ and $\delta_1 \in (0, \delta_0)$ such that for every $\delta \in (0, \delta_1)$, 
$\textup{diam}(\textup{int}(B(x,\delta))) \le \Phi_d (\delta, \delta) \le C_1 \delta$. 
By Lemma \ref{lem:interior-open} (i), $\textup{int}(B(x,\delta))$ is open.
By Lemma \ref{lem:interior-open} (ii), there exists $\eta > 0$ such that $B(x,\eta) \subset \textup{int}(B(x,\delta))$. 
Then by Proposition \ref{prop:mass-dist}, 
\[ \mu_{\bf p}(\{x\}) \le \mu_{\bf p} \left( \textup{int}(B(x,\delta)) \right) \le C_0 (C_1 \delta)^{\alpha}.   \]
By letting $\delta \to +0$, we have the assertion. 
\end{proof}

We remark that a ball may not be a Borel set, so $\mu_{\bf p}(B(x,\eta))$ might not be well-defined. 

We need the following lemma since a ball may not be open. 

\begin{Lem}\label{lem:large-open}
There exist $C_2 > 0$ and $\delta_2 \in (0, \delta_1)$  such that for every subset $A$ of $X$ with $0 < \textup{diam}(A) < \delta_2$, 
there exists an open subset $U$ such that $A \subset U$ and $ \textup{diam}(U) \le C_2 \textup{diam}(A)$. 
\end{Lem}

\begin{proof}
Assume that $0 < \textup{diam}(A) < \infty$. 
Since $A \ne \emptyset$, we can take a point $x \in A$. 
Then, by using $\textup{diam}(A) > 0$, we obtain that $A \subset B(x, 2\textup{diam}(A))$. 
Take $C_1 > 0$ and $\delta_1$ as in the proof of Corollary \ref{cor:non-atomic}. 

If $0 < \textup{diam}(A) < \delta_1 / (4(C_1 + 1))$, then, $\Phi_d (2\textup{diam}(A), 2\textup{diam}(A)) \le 2C_1 \textup{diam}(A)$. 
By using the same argument as in the proof of Lemma \ref{lem:interior-open} (i), 
$B(x, 2\textup{diam}(A)) \subset \textup{int} \left(B(x, 4C_1 \textup{diam}(A)) \right)$. 
Then 
\[ \textup{diam}\left(  \textup{int} \left(B(x, 4C_1 \textup{diam}(A)) \right) \right) \le \Phi_d (4C_1 \textup{diam}(A),  4C_1 \textup{diam}(A)) \le 4C_1^2 \textup{diam}(A). \]
Now the assertion holds for $\delta_2 =  \delta_1 / (4(C_1 + 1))$, $C_2 = 4C_1^2$, and $U =  \textup{int} \left(B(x, 4C_1 \textup{diam}(A)) \right)$. 
\end{proof}

\begin{Thm}\label{thm:main}
Suppose that  Assumptions \ref{ass:two-additional-metric} and \ref{ass:additional-f}  hold. 
Let $K$ be the attractor of $\{f_j\}_j$. 
Let $\alpha$ be the similarity dimension of $\{f_j\}_j$, specifically,  a positive real number such that $\sum_j r_j^{\alpha} = 1$. 
Then $0 < \mathcal{H}^{\alpha}(K) < \infty$ and $\dim_H K = \dim_M K = \alpha$. 
\end{Thm}

\begin{proof}
We show that $ \mathcal{H}^{\alpha}(K) > 0$. 
The following arguments correspond to the proof of the mass distribution principle, but we need several modifications. 

Let $\delta \in (0, \delta_2 /(C_2 + 1))$. 
Assume that $K \subset \bigcup_{i \ge 1} A_i$ and that $\textup{diam}(A_i) < \delta$ for every $i$. 
Let $\mathcal{I} \coloneqq \{i \ge 1 \colon \textup{diam}(A_i) = 0\}$.   
By Lemma \ref{lem:large-open}, 
it holds that for every $i \notin \mathcal{I}$, there exists an open set $U_i$ such that $A_i \subset U_i$ and $\textup{diam}(U_i) \le C_2 \textup{diam}(A_i)$. 
Since 
\[ K \setminus \bigcup_{i \in \mathcal{I}} A_i \subset  \bigcup_{i \ge 1, i \notin \mathcal{I}} A_i \subset  \bigcup_{i \ge 1, i \notin \mathcal{I}} U_i, \]
we obtain that 
\[ \mu_{\bf p}\left( K \setminus \bigcup_{i \in \mathcal{I}} A_i  \right) \le \sum_{i \ge 1, i \notin \mathcal{I}} \mu_{\bf p}(U_i). \]

We remark that for every $i \in \mathcal{I}$, $A_i$ is empty or a one-point set. 
By Corollary \ref{cor:non-atomic}, $ \mu_{\bf p}\left( K \setminus \bigcup_{i \in \mathcal{I}} A_i  \right)  =  \mu_{\bf p}\left( K \right)  = 1$. 
Since $ \textup{diam}(U_i) \le C_2 \textup{diam}(A_i) < \delta_2 < \delta_0$, 
we can apply Proposition \ref{prop:mass-dist} and obtain that for $i \notin \mathcal{I}$, 
\[ \mu_{\bf p}(U_i) \le C_0  \textup{diam}(U_i) ^{\alpha} \le C_0 C_2^{\alpha} \textup{diam}(A_i)^{\alpha}.\] 
Then, 
\[ \sum_{i \ge 1} \textup{diam}(A_i)^{\alpha} = \sum_{i \ge 1, i \notin \mathcal{I}} \textup{diam}(A_i)^{\alpha} \ge \frac{1}{C_0 C_2^{\alpha}} \sum_{i \ge 1, i \notin \mathcal{I}} \mu_{\bf p}(U_i) \ge  \frac{1}{C_0 C_2^{\alpha}}.  \]
Hence $\mathcal{H}^{\alpha}_{\delta}(K) \ge  \frac{1}{C_0 C_2^{\alpha}}$ for every $\delta \in (0,\delta_2 /(C_2 + 1))$ and we obtain that $\mathcal{H}^{\alpha}(K) \ge  \frac{1}{C_0 C_2^{\alpha}}$. 

We see that $\mathcal{H}^{\alpha}(K) = \mathcal{H}^{\alpha}_{\infty}(K)  < \infty$ in the same manner as in the proof of Lemma \ref{lem:open-Haus-basic} below. 
Furthermore, we also see that $\overline{\dim}_M K \le \alpha$ by following the proof of \cite[Theorem 2.2.2 (ii)]{BishopPeres2017}. 
\end{proof}

\begin{Rem}
Conversely, if  Assumptions \ref{ass:two-additional-metric} and \ref{ass:additional-f}  hold and $0 < \mathcal{H}^{\beta}(K) < \infty$ for some $\beta > 0$, 
then $\sum_{j} r_j^{\beta} \ge 1$. 
If, in addition, $\mathcal{H}^{\beta}(f_k (K) \cap f_{\ell}(K)) = 0$ for all $k \ne \ell$, then $\sum_{j} r_j^{\beta} = 1$. 
See \cite[Theorem 3]{BessenyeiPenzes2022}\footnote{In the second statement of \cite[Theorem 3]{BessenyeiPenzes2022}, the assumption that each $f_j$ is surjective is missing.}. 
\end{Rem}

\section{The open set condition is necessary}\label{sec:osc}

In this section, we extend the necessity result proved in \cite{Schief1994} from the Euclidean setting to the setting of complete, normal and regular semimetric spaces.  

\begin{Thm}\label{thm:osc-necessary}
Suppose that Assumptions \ref{ass:two-additional-metric} and \ref{ass:additional-f} (ii) hold. 
Let $K$ be the attractor of $\{f_j\}_j$. 
Let $\alpha$ be the similarity dimension of the similitudes $\{f_j\}_j$. 
If $\mathcal{H}^{\alpha}(K) > 0$, then the open set condition holds.  
\end{Thm}

As an outline, we follow the proof in \cite[Section 9.6]{BishopPeres2017}, which is a simplification of  \cite{Schief1994}. 
We make only a few departures, but several parts are a little more complicated.

We remark that the diameter of the $\epsilon$-neighborhood of $A$ is hard to estimate in terms of the diameter of $A$. 

For $A  \subset X$, $\beta \ge 0$ and $\delta \in (0, \infty]$, let 
\[ \mathcal{H}^{\beta}_{o;\delta}(A) \coloneqq \inf\left\{\sum_i \textup{diam}(U_i)^{\beta}  \colon A \subset \bigcup_i U_i, \ U_i \textup{ open}, \ \textup{diam}(U_i) \le \delta \right\}  \] 
and 
\[ \mathcal{H}^{\beta}_{o}(A) \coloneqq \lim_{\delta \to +0} \mathcal{H}^{\beta}_{o;\delta}(A). \] 
Then $\mathcal{H}^{\beta}_{o;\delta}(A) \ge \mathcal{H}^{\beta}_{\delta}(A)$ and hence $\mathcal{H}^{\beta}_{o}(A) \ge \mathcal{H}^{\beta}(A)$. 
$\mathcal{H}^{\beta}_{o}$ is an outer measure on $X$. 
By the same arguments as in the proof of  \cite[Lemma 11]{BessenyeiPenzes2022}, all Borel subsets of $X$ are $\mathcal{H}^{\beta}_{o}$-measurable. 

We say that two finite strings $\sigma, \tau \in \{1,\dots, \ell\}^{*}$ are {\it incomparable} if neither is a prefix of the other.  
The following corresponds to \cite[Proposition 2.1.3]{BishopPeres2017}. 

\begin{Lem}\label{lem:open-Haus-basic}
Let $\alpha$ be a positive real number such that $\sum_j r_j^{\alpha} = 1$. 
Then,\\
(i) $\mathcal{H}^{\alpha}_{o;\infty}(K) = \mathcal{H}^{\alpha}_{o}(K) < \infty$. \\
(ii) For every $ \mathcal{H}^{\alpha}_{o}$-measurable subset $B \subset K$, $\mathcal{H}^{\alpha}_{o;\infty}(B) = \mathcal{H}^{\alpha}_{o}(B)$. \\
(iii) If $\sigma$ and $\tau$ are incomparable, then $\mathcal{H}^{\alpha}_{o} \left(f_{\sigma}(K) \cap f_{\tau}(K) \right) = 0$. 
\end{Lem}

\begin{proof}
(i) Since $K \subset \bigcup_{x \in K} \textup{int}(B(x,1))$ and $K$ is compact, 
there exist $x_1, \dots, x_m \in K$ such that $K \subset \bigcup_{k=1}^{m} \textup{int}(B(x_k, 1))$. 
Since $X$ is normal, the diameter of a ball is finite. 
Hence $\mathcal{H}^{\alpha}_{o;\infty}(K) < \infty$. 

It is obvious that $\mathcal{H}^{\alpha}_{o;\infty}(K) \le \mathcal{H}^{\alpha}_{o}(K)$. 
We show the converse. 
We recall that $\mathcal{H}^{\alpha}_{o;\infty}(K) < \infty$. 
Let $\{U_i\}_i$ be a countable family of open sets covering $K$ such that $\sum_i \textup{diam}(U_i)^{\alpha} \le \mathcal{H}^{\alpha}_{o;\infty}(K) +\epsilon$.  
Take $n \in \mathbb{N}$ so large that $r_{\max}^n \sup_i \textup{diam}(U_i) < \epsilon$. 
By Lemma \ref{lem:ball-surjection}, each $f_j$ is an open map. 
Then $\{f_{\sigma}(U_i) \colon  |\sigma| = n, i \ge 1\}$ is an open cover of $K$ such that $\textup{diam}(f_{\sigma}(U_i)) < \epsilon$. 
By using the definition of $\alpha$, 
\[ \mathcal{H}_{o;\epsilon}^{\alpha}(K) \le \sum_{i} \sum_{\sigma : |\sigma| = n} \textup{diam}(f_{\sigma}(U_i))^{\alpha} =   \sum_{i} \textup{diam} (U_i)^{\alpha}  \le \mathcal{H}^{\alpha}_{o;\infty}(K) +\epsilon. \]
By letting $\epsilon \to +0$, we obtain that $\mathcal{H}^{\alpha}_{o;\infty}(K) \ge \mathcal{H}^{\alpha}_{o}(K)$. 
Thus we see that $\mathcal{H}^{\alpha}_{o;\infty}(K) = \mathcal{H}^{\alpha}_{o}(K)$.

(ii) By subadditivity of outer measures, 
\[ \mathcal{H}^{\alpha}_{o;\infty}(K) \le \mathcal{H}^{\alpha}_{o;\infty}(B) + \mathcal{H}^{\alpha}_{o;\infty}(K \setminus B) \le  \mathcal{H}^{\alpha}_{o}(B) + \mathcal{H}^{\alpha}_{o}(K \setminus B). \]
Since $B$ is $ \mathcal{H}^{\alpha}_{o}$-measurable, 
$\mathcal{H}^{\alpha}_{o}(B) + \mathcal{H}^{\alpha}_{o}(K \setminus B) = \mathcal{H}^{\alpha}_{o}(K)$. 
By (i), $ \mathcal{H}^{\alpha}_{o}(K) =  \mathcal{H}^{\alpha}_{o;\infty}(K)$. 
Since $ \mathcal{H}^{\alpha}_{o}$ is always larger than or equal to $ \mathcal{H}^{\alpha}_{o;\infty}$, 
we see that $\mathcal{H}^{\alpha}_{o;\infty}(B) = \mathcal{H}^{\alpha}_{o}(B)$. 

(iii) Since each $f_j$ is continuous, $f_j (K)$ is also compact. 
Since $X$ is Hausdorff, $f_j (K)$ is closed. 
Furthermore, for every $\sigma \in \{1,\dots, \ell\}^{*}$ and $A \subset X$, 
$\mathcal{H}^{\alpha}_{o}(f_{\sigma}(A)) = r_{\sigma}^{\alpha} \mathcal{H}^{\alpha}_{o}(A)$. 
Then 
\[ \mathcal{H}^{\alpha}_{o}(K) =  \mathcal{H}^{\alpha}_{o} \left(\bigcup_{j} f_j(K) \right)  \le \sum_{j}  \mathcal{H}^{\alpha}_{o}(f_j(K)) = \sum_j r_j^{\alpha}  \mathcal{H}^{\alpha}_{o}(K) =  \mathcal{H}^{\alpha}_{o}(K).   \]
Hence, $ \mathcal{H}^{\alpha}_{o} \left(\bigcup_{j} f_j(K) \right) = \sum_{j}  \mathcal{H}^{\alpha}_{o}(f_j(K))$. 
Since $ \mathcal{H}^{\alpha}_{o}$ is additive on the Borel subsets of $X$, 
$\mathcal{H}^{\alpha}_{o}(f_{i}(K) \cap f_{j}(K)) = 0$ for $i \ne j$. 

If $\sigma$ and $\tau$ are incomparable, then there exist $i \ne j$ such that $f_{\sigma}(K) \subset f_{\sigma \wedge \tau} (f_i (K))$ and $f_{\tau}(K) \subset f_{\sigma \wedge \tau} (f_j (K))$, where $\sigma \wedge \tau$ is the longest common prefix of $\sigma$  and $\tau$. 
 Since $f_{\sigma \wedge \tau} $ is a bijection, 
 \[ \mathcal{H}^{\alpha}_{o} \left(f_{\sigma}(K) \cap f_{\tau}(K) \right) \le \mathcal{H}^{\alpha}_{o} \left(f_{\sigma \wedge \tau}(f_i(K)) \cap f_{\sigma \wedge \tau}(f_j(K)) \right) \]
 \[ =  \mathcal{H}^{\alpha}_{o} \left(f_{\sigma \wedge \tau}(f_i(K) \cap f_j(K)) \right) = r_{\sigma \wedge \tau}^{\alpha} \mathcal{H}^{\alpha}_{o}(f_{i}(K) \cap f_{j}(K))  = 0. \]
\end{proof}

Since $K$ is compact and $\mathcal{H}_o^{\alpha}(K) \ge \mathcal{H}^{\alpha}(K) > 0$, there exist open sets $U_1, \dots, U_n$ such that $K \subset \bigcup_{i=1}^{n} U_i$ and $\sum_{i} \textup{diam}(U_i)^{\alpha} \le (1+r_{\min}^{\alpha})\mathcal{H}_{o}^{\alpha}(K)$. 
Let $U \coloneqq  \bigcup_{i=1}^{n} U_i$ and $\delta \coloneqq \textup{dist}(K, U^c)$. 
We show that $\delta > 0$. 
Assume that $\delta = 0$. 
Then there exist two sequences $(x_n)_n$ in $K$ and $(y_n)_n$ in $U^c$ such that $d(x_n, y_n) \to 0, n \to \infty$. 
Since $K$ is compact, it is sequentially compact and hence there exists a subsequence $(x_{k_n})_n$ and a limit $x \in K$ such that $d(x_{k_n}, x) \to 0, n \to \infty$. 
Since $X$ is regular, $d(x, y_{k_n}) \le \Phi_d \left(d(x, x_{k_n}), d(x_{k_n}, y_{k_n})\right) \to 0, n \to \infty$. 
By Lemma \ref{lem:seq-closure}, $x \in U^c$, which contradicts $K \subset U$. 

First, we obtain a separation result for several pairs of $\left\{f_{\sigma}(K) \colon \sigma \in \{1, \dots, \ell\}^{*} \right\}$ with respect to the Hausdorff distance. 
Define the Hausdorff distance between subsets of $X$ by 
\[ d_H (A_1, A_2) \coloneqq \inf\left\{\epsilon > 0 \colon A_1 \subset \bigcup_{y \in A_2} B(y,\epsilon), \  A_2 \subset \bigcup_{x \in A_1} B(x,\epsilon) \right\}, \quad A_1, A_2 \subset X. \]

\begin{Lem}\label{lem:sep-dist-Haus}
 If $\sigma$ and $\tau$ are incomparable and $r_{\tau} > r_{\min} r_{\sigma}$, 
 then 
 \[ d_H (f_{\sigma}(K), f_{\tau}(K)) \ge \delta r_{\sigma}.\]  
\end{Lem}

\begin{proof}
We show this by contradiction. 
Assume that $d_H (f_{\sigma}(K), f_{\tau}(K)) < \delta r_{\sigma}$. 
We first remark that $f_{\sigma}(K) \subset f_{\sigma}(U)$. 
Let $\eta \in (d_H (f_{\sigma}(K), f_{\tau}(K)), \delta r_{\sigma})$. 
Since $\textup{dist}(f_{\sigma}(K), (f_{\sigma}(U))^c) = \textup{dist}(f_{\sigma}(K), f_{\sigma}(U^c)) = \delta r_{\sigma}$, 
\[ f_{\tau}(K) \subset \bigcup_{x \in f_{\sigma}(K)} B(x,\eta) \subset f_{\sigma}(U).  \]

Therefore, $\{f_{\sigma}(U_i) \colon 1 \le i \le n\}$ is an open cover of $f_{\sigma}(K) \cup f_{\tau}(K)$. 
Since $K$ is the attractor, $f_{\sigma}(K) \cup f_{\tau}(K) \subset K$ and $f_{\sigma}(K) \cup f_{\tau}(K)$ is a closed subset of $X$. 
We can apply Lemma \ref{lem:open-Haus-basic} (ii) and obtain that 
\[ \mathcal{H}^{\alpha}_{o} \left(f_{\sigma}(K) \cup f_{\tau}(K) \right) = \mathcal{H}^{\alpha}_{o;\infty} \left(f_{\sigma}(K) \cup f_{\tau}(K) \right) \le \sum_{i=1}^{n} \textup{diam}(f_{\sigma}(U_i))^{\alpha} \]
\[ = r_{\sigma}^{\alpha} \sum_{i} \textup{diam}(U_i)^{\alpha} \le  r_{\sigma}^{\alpha}(1+r_{\min}^{\alpha})\mathcal{H}_{o}^{\alpha}(K). \]
By Lemma \ref{lem:open-Haus-basic} (iii), 
\[ \mathcal{H}^{\alpha}_{o} \left(f_{\sigma}(K) \cup f_{\tau}(K) \right) = \mathcal{H}^{\alpha}_{o} \left(f_{\sigma}(K) \right) + \mathcal{H}^{\alpha}_{o} \left(f_{\tau}(K) \right) =  (r_{\sigma}^{\alpha} +r_{\tau}^{\alpha})\mathcal{H}_{o}^{\alpha}(K). \]
By the assumption that $r_{\tau} > r_{\min} r_{\sigma}$ and $0 < \mathcal{H}^{\alpha}_{o}(K) < \infty$, 
we see that there is a contradiction. 
\end{proof}

We remark that $\textup{diam}(K) > 0$ since $\mathcal{H}^{\alpha}(K) > 0$. 
For $b > 0$, we let 
\begin{equation}\label{eq:def-cutset-2} 
\Pi_b \coloneqq \left\{\sigma \in \{1,\dots,\ell\}^{*} \colon r_{\sigma} \textup{diam}(K)  \le b < r_{\sigma^{\prime}} \textup{diam}(K)  \right\}. 
\end{equation}
This is a minimal cut-set if $b < \textup{diam}(K)$, which is equivalent to $\Pi_b \ne \emptyset$. 

For $\epsilon \in (0,1)$, let $G_{\epsilon} \coloneqq \bigcup_{x \in K} B(x,\epsilon)$. 
By the normality of $X$, $\textup{diam}(G_{\epsilon}) < \infty$. 
However it is not simple to relate the diameter of $G_{\epsilon}$ to the diameter of $K$. 

For $v \in \{1,\dots, \ell\}^{*}$, let 
\[ \Gamma_{\epsilon}(v) \coloneqq \left\{\sigma \in \Pi_{\textup{diam}(f_{v} (G_{\epsilon}))} \colon f_{\sigma}(K) \cap f_{v} (G_{\epsilon}) \ne \emptyset \right\}.\]

\begin{Lem}\label{lem:doubling-closure-ball-cpt}
If  the geometric doubling property in Assumption \ref{ass:two-additional-metric} (ii) holds, 
then the closure of a ball is compact. 
\end{Lem}

\begin{proof}
Let $x \in X$ and $r > 0$. 
Let\footnote{It might hold that $r^* > \Phi_d (r,0)$.} $r^* \coloneqq \lim_{\delta \to +0} \Phi_d (r,\delta)$. 
Let $y \in \overline{B(x,r)}$. 
By Lemma \ref{lem:seq-closure}, there exists a sequence $(x_n)_n$ in $B(x,r)$ such that $d(x_n, y) \to 0, n \to \infty$.  
Since $d(x,y) \le \Phi_d (d(x, x_n), d(x_n, y)) \le \Phi_d (r, d(x_n, y))$ for every $n$, 
$d(x,y) \le r^*$. 
Hence, $ \overline{B(x,r)} \subset B(x, r^* + 1)$. 
By  the geometric doubling property, $\overline{B(x,r)}$ is totally bounded. 
Since $X$ is complete,  $\overline{B(x,r)}$ is also complete. 
Thus we see that  $\overline{B(x,r)}$ is compact. 
\end{proof}

\begin{Lem}\label{lem:set-of-closed-is-cpt-wrt-Haus-dist}
Let $E$ be a nonempty compact subset of $X$. 
Let $\mathbf{C}(E)$  be the set of nonempty closed subsets of $E$. 
Then, $(\mathbf{C}(E), d_H)$ is a totally bounded regular semimetric space. 
\end{Lem}

\begin{proof}
By \cite[Lemma 7]{BessenyeiPenzes2022}, 
$(\mathbf{C}(E), d_H)$ is a regular semimetric space. 
Let $\epsilon > 0$. 
Then there exist $x_1, \dots, x_n \in E$ such that $E \subset \bigcup_{i=1}^{n} B(x_i, \epsilon)$. 
Let $S \coloneqq \{x_1, \dots, x_n\}$. 
Let $\mathcal{S}$ be the set of nonempty subsets of $S$. 
Then $\mathcal{S}$ is a subset of $\mathbf{C}(E)$ since $X$ is Hausdorff. 
Let $C \in \mathbf{C}(E)$. 
Let $A \coloneqq \{x \in S \colon \textup{ there exists $y \in C$ such that $d(x,y) < \epsilon$} \}$. 
Then $A \subset \bigcup_{y \in C} B(y,\epsilon)$ and $C \subset \bigcup_{x \in A} B(x,\epsilon)$. 
Hence, $d_H (A, C) \le \epsilon$. 
We remark that $A \in \mathcal{S}$ and $C \in B_{d_H}(A, 2\epsilon)$. 
Thus $\mathcal{S}$ is a $2\epsilon$-net in $ \mathbf{C}(E)$. 
\end{proof}

\begin{Lem}\label{lem:inverse-contained}
There exists a constant $R_0$ depending only on $\textup{diam}(K)$ such that for every $x \in K$, $v \in \{1,\dots, \ell\}^{*}$ and $\sigma \in \Pi_{\epsilon}(v)$, 
$f_v^{-1}\left(f_{\sigma}(K)\right) \subset B(x,R_0)$. 
\end{Lem}

\begin{proof}
Let $x \in K$. 
Take a point $z \in f_{\sigma}(K) \cap f_{v}(G_{\epsilon})$. 
Then, for every $y \in f_{\sigma}(K)$, 
\[ d(f_v^{-1}(y), x) \le \Phi_d \left( d(f_v^{-1}(y), f_v^{-1}(z)),  d(f_v^{-1}(z), x)\right). \]
We see that 
\[ d(f_v^{-1}(y), f_v^{-1}(z)) \le r_v^{-1} r_{\sigma} \textup{diam}(K) \le \textup{diam}(G_{\epsilon}) \le \Phi_d (\epsilon, \Phi_d (\epsilon, \textup{diam}(K))) \]
and $d(f_v^{-1}(z), x) \le \Phi_d (\textup{diam}(K), \epsilon)$. 
Recall that $0 < \epsilon < 1$. 
Let $R_1 \coloneqq \Phi_d (\textup{diam}(K), 1)$ and $R_2 \coloneqq R_1 + \Phi_d (1, R_1)$. 
Then, 
\[ d(f_v^{-1}(y), x) \le \Phi_d \left( \Phi_d (1, R_1),  R_1 \right) \le \Phi_d (R_2, R_2). \]
Now the assertion holds if $R_0 \coloneqq 1 + \Phi_d (R_2, R_2)$. 
\end{proof}

\begin{Prop}
$\sup_{v \in \{1, \dots, \ell\}^{*}} \left|\Gamma_{\epsilon}(v) \right| < \infty$.
\end{Prop}

\begin{proof}
Let $v \in \{1, \dots, \ell\}^{*}$. 
Let $\sigma, \tau \in \Gamma_{\epsilon}(v)$ and $\sigma \ne \tau$. 
Then, $\sigma, \tau \in \Pi_{r_{v} \textup{diam}(G_{\epsilon})}$. 
Hence $\sigma$ and $\tau$ are incomparable and $r_{\tau} > r_{\sigma} r_{\min}$. 
Hence, by Lemma \ref{lem:sep-dist-Haus}, 
$d_H (f_{\sigma}(K), f_{\tau}(K)) \ge \delta r_{\sigma}$. 
Therefore, 
\begin{equation}\label{eq:sep-Haus-different-strings}
d_H \left(f_v^{-1}(f_{\sigma}(K)), f_v^{-1}(f_{\tau}(K)) \right) \ge \delta \frac{r_{\sigma}}{r_{v}}  \ge \delta r_{\min} \frac{r_{\sigma^{\prime}}}{r_{v}} > \delta r_{\min}  \frac{\textup{diam}(G_{\epsilon})}{\textup{diam}(K)} \ge \delta r_{\min}. 
\end{equation} 

Let $x \in K$. 
Then by Lemma \ref{lem:doubling-closure-ball-cpt}, $\overline{B(x,R_0)}$ is compact. 
By Lemma \ref{lem:inverse-contained}, $f_v^{-1} (f_{\sigma}(K)) \subset B(x,R_0)$. 
Since $f_v$ is a homeomorphism, 
$f_v^{-1} (f_{\sigma}(K))$ is compact and closed in $X$ for every $\sigma$. 
Hence $f_v^{-1} (f_{\sigma}(K)) \in \mathbf{C}\left(\overline{B(x,R_0)} \right)$. 

By Lemma \ref{lem:set-of-closed-is-cpt-wrt-Haus-dist}, $\left(\mathbf{C}\left(\overline{B(x,R_0)}\right), d_H \right)$ is a totally bounded regular semimetric space. 
By the regularity of $d_H$, 
there exists $\eta = \eta\left( \delta r_{\min} \right) > 0$ such that $\textup{diam}(B_{d_H} (Z, \eta)) < \frac{\delta r_{\min}}{2}$ for every $Z \in  \mathbf{C}\left(\overline{B(x,R_0)} \right)$. 
Furthermore, there exist $N = N_{\eta}$ and $A_1, \dots, A_N \in  \mathbf{C}\left(\overline{B(x,R_0)} \right)$ such that  $ \mathbf{C}\left(\overline{B(x,R_0)} \right) = \bigcup_{i=1}^{N} B_{d_H}(A_i, \eta)$. 
By \eqref{eq:sep-Haus-different-strings}, each $B_{d_H}(A_i, \eta)$ contains at most one of  $f_v^{-1} (f_{\sigma}(K))$ and $f_v^{-1} (f_{\tau}(K))$ if $\sigma \ne \tau$. 
Thus we see that $ \left|\Gamma_{\epsilon}(v) \right|  \le N_{\eta}$. 
\end{proof}

Since $\left|\Gamma_{\epsilon}(v) \right|$ is an integer, there exists $v_0 \in \{1,\dots, \ell\}^{*}$ such that $\left|\Gamma_{\epsilon}(v_0) \right| = \sup_{v \in \{1, \dots, \ell\}^{*}} \left|\Gamma_{\epsilon}(v) \right|$. 
We remark that $\Gamma_{\epsilon}(v) \ne \emptyset$ for some $v \in \{1, \dots, \ell\}^{*}$.  
Indeed, there exists $v \in \{1,\dots, \ell\}^{*}$ such that $r_{v} \textup{diam}(G_{\epsilon}) < \textup{diam}(K)$. 
Then, there exists a prefix $v^{\prime}$ of $v$ such that $v^{\prime} \in \Pi_{\textup{diam}(f_{v}(G_{\epsilon}))}$. 
Since $K$ is the attractor,  $f_{v^{\prime}}(K) \cap f_{v}(G_{\epsilon}) \supset f_{v}(K) \ne \emptyset$.  
Hence $\left|\Gamma_{\epsilon}(v_0) \right| \ge 1$. 

Let $\tau \in \Gamma_{\epsilon}(v_0)$. 
Then for every $\sigma \in \{1,\dots, \ell\}^{*}$, $\sigma \tau \in \Gamma_{\epsilon}(\sigma v_0)$. 
Hence $ \{\sigma \tau \colon \tau \in \Gamma_{\epsilon}(v_0)\} \subset \Gamma_{\epsilon}(\sigma v_0)$. 
By the maximality, 
\[ \Gamma_{\epsilon}(\sigma v_0) = \{\sigma \tau \colon \tau \in \Gamma_{\epsilon}(v_0)\}.\]

We now obtain a separation result for several pairs of $\{f_{\sigma}(K) \colon \sigma \in \{1, \dots, \ell\}^{*} \}$ with respect to the distance between subsets of $X$. 
Let $i \ne j$, $\sigma \in \{1, \dots, \ell\}^{*}$ and $j \tau \in \Pi_{\textup{diam}(f_{i \sigma v_0}(G_{\epsilon}))}$. 
Since $j \tau \notin \Gamma_{\epsilon}(i\sigma v_0)$, $f_{j \tau}(K) \cap f_{i \sigma v_0}(G_{\epsilon}) = \emptyset$. 
We can show that  
\[ \textup{dist}\left(f_{j \tau}(K), f_{i \sigma v_0}(K) \right) \ge \epsilon r_{i \sigma v_0}.\]  
Indeed, if this failed, then there would exist $x_1, x_2 \in K$ such that $d(f_{j \tau}(x_1), f_{i \sigma v_0}(x_2)) < \epsilon r_{i \sigma v_0}$. 
Since $ f_{i \sigma v_0}$ is a similitude with rate $r_{i \sigma v_0}$, 
$d\left(f_{i \sigma v_0}^{-1}(f_{j \tau}(x_1)), x_2 \right) < \epsilon$. 
By the definition of $G_{\epsilon}$, $f_{i \sigma v_0}^{-1}(f_{j \tau}(x_1)) \in G_{\epsilon}$ and hence $f_{j \tau}(x_1) \in f_{j \tau}(K) \cap f_{i \sigma v_0}(G_{\epsilon})$. 
This is a contradiction. 
Since $\left|\Gamma_{\epsilon}(v_0) \right| \ge 1$, $\Pi_{r_{v_0} \textup{diam}(G_{\epsilon})} \ne \emptyset$ and hence 
\[ \textup{diam}(f_{i \sigma v_0}(G_{\epsilon})) = r_{i \sigma v_0} \textup{diam}(G_{\epsilon}) < r_{v_0} \textup{diam}(G_{\epsilon}) < \textup{diam}(K).\]
Hence $\Pi_{\textup{diam}(f_{i \sigma v_{0}}(G_{\epsilon}))}$ is a cut-set. 
Thus we see that 
\[ f_j (K) = \bigcup_{j\tau \in \Pi_{\textup{diam}(f_{i \sigma v_{0}}(G_{\epsilon}))}} f_{j\tau}(K). \] 
Hence, 
\begin{equation}\label{eq:separation-set-dist}  
\textup{dist}\left(f_{j}(K), f_{i \sigma v_{0}}(K) \right) \ge \epsilon r_{i \sigma v_{0}}. 
\end{equation}

By Assumption \ref{ass:two-additional-metric} (i), 
there exist $R_0 \in (0,1)$ and $C_0 > 1$ such that for every $t \in [0, R_0]$, $\Phi_d (t,t) \le C_0 t$. 
Let $\epsilon \in (0, R_0 /2)$. 
Let 
\[ G^* \coloneqq \bigcup_{x \in K} \textup{int} \left(B\left(x, \frac{\epsilon}{2C_0} \right)\right).\] 
Take $v_0 \in \{1, \dots, \ell\}^{*}$ such that  $\left|\Gamma_{\epsilon}(v_0) \right| = \sup_{v \in \{1, \dots, \ell\}^{*}} \left|\Gamma_{\epsilon}(v) \right|$. 
Let 
\[ V \coloneqq \bigcup_{\sigma \in \{1, \dots, \ell\}^{*}} f_{\sigma v_0} \left(G^{*} \right). \]  
Then $V$ is a nonempty open subset of $X$. 
Let $x \in V$. 
Then there exist $\sigma \in \{1, \dots, \ell\}^{*}$ and $y \in G^*$ such that $x = f_{\sigma v_{0}} (y)$. 
There exists $z \in K$ such that $d(y,z) < \epsilon/(2C_0)$. 
Then $d(x, f_{\sigma v_{0}}(z)) < r_{\sigma v_{0}}  \epsilon/(2C_0) < R_0 / (4 C_0)$. 
Hence $x \in G_{R_0 / (4 C_0)}$. 
Thus we see that $V \subset G_{R_0 / (4 C_0)}$; in particular, $V$ is bounded. 

Let $x \in f_i (V)$. 
Then there exists $\sigma  \in \{1, \dots, \ell\}^{*}$ such that $x \in f_i (f_{\sigma v_0}(G^{*})) = f_{\tau v_0}(G^{*}) \subset V$ for $\tau = i \sigma$. 
Hence $f_i (V) \subset V$. 

Finally we show that $f_i (V) \cap f_j (V) = \emptyset$ for $i \ne j$. 
We show this by contradiction. 
Let $y \in f_i (V) \cap f_j (V)$. 
Then there exist $\sigma, \tau \in \{1, \dots, \ell\}^{*}$ such that $y \in f_{i \sigma v_0} (G^*) \cap f_{j \tau v_0} (G^*)$. 
There exist $y_1 \in f_{i \sigma v_0}(K)$ and $y_2 \in f_{j \tau v_0}(K)$ such that $d(y_1, y) < r_{i \sigma v_0} \epsilon /(2 C_0)$ and $d(y_2, y) < r_{j \tau v_0} \epsilon /(2 C_0)$. 
By the strong regularity in Assumption \ref{ass:two-additional-metric} (i), 
\[ d(y_1, y_2) \le \Phi_d (d(y_1, y), d(y_2, y)) \le C_0 \max\{d(y_1, y), d(y_2, y)\} \le \frac{\epsilon}{2} \max\left\{r_{i \sigma v_0}, r_{j \tau v_0} \right\}. \]
Hence $\textup{dist}\left(f_{i \sigma v_0}(K),  f_{j \tau v_0}(K)\right) < \epsilon \max\{r_{i \sigma v_0}, r_{j \tau v_0} \}$. 
This contradicts \eqref{eq:separation-set-dist}. 

Thus the proof of Theorem \ref{thm:osc-necessary} is completed. 

\begin{Rem}
(i) Lemma \ref{lem:inverse-contained} corresponds to \cite[(9.6.3)]{BishopPeres2017}; however, the proof is a little different. \\
(ii) We have {\it not} assumed that $\textup{diam}(K) = 1$ since we cannot guarantee that $\textup{diam}(G_{\epsilon}) = 1+ O(\epsilon), \ \epsilon \to +0$. 
This is the reason for introducing \eqref{eq:def-cutset-2} instead of \eqref{eq:def-cutset-1}. 
\end{Rem}

\section{Examples}\label{sec:exa}

\begin{Exa}
We consider the case where $X = \mathbb{R}$ and $d(x, x^{\prime}) = |x-x^{\prime}|^{\beta}$ for some $\beta > 0$. 
The pair $(X,d)$ is a  semimetric space and is a metric space if $\beta \le 1$. 
It holds that $B(x,r) = B_{\textup{Euc}}(x,r^{1/\beta})$ for every $x \in X$ and $r > 0$, where $B_{\textup{Euc}}$ denotes a ball  under the Euclidean distance. 
Hence $(X,d)$ is complete and Assumption \ref{ass:two-additional-metric} (ii) holds.
Furthermore, it holds that for $A \subset X$, 
$\mathcal{H}^r (A) = \mathcal{H}^{\beta r}_{\textup{Euc}}(A)$ and hence $\dim_H (A) = \beta^{-1} \dim_H^{\textup{Euc}} (A)$
where $\mathcal{H}_{\textup{Euc}}$ is the Hausdorff measure and $\dim_H^{\textup{Euc}}$ is the Hausdorff dimension with respect to the Euclidean distance.

By subadditivity for $0 < \beta \le 1$ and convexity for $\beta \ge 1$, 
$(x+ y)^{\beta} \le \max\{1,2^{\beta -1}\} (x^{\beta} + y^{\beta})$ for $x, y \ge 0$. 
Let $\Phi  (u,v) \coloneqq \max\{1,2^{\beta-1}\} (u+v)$. 
Then this is a triangle function on $(X,d)$, which is continuous at $(0,0)$. 
Therefore the semimetric space $(X,d)$ is normal and regular and satisfies Assumption \ref{ass:two-additional-metric} (i). 

Let $r_1, r_2 \in (0,1)$ and $f_1 (x) \coloneqq r_1^{1/\beta} x$ and $f_2 (x) \coloneqq r_2^{1/\beta} x + 1 - r_2^{1/\beta}$. 
Then, $d(f_j (x), f_j (y)) = r_j d(x,y)$ for $x, y \in \mathbb{R}$ and $ j \in \{1,2\}$. 
In particular each $f_j$ is a similitude. 
By Theorem \ref{thm:attractor}, there exists a unique attractor $K$ of $\{f_j\}_j$.

If $r_1^{1/\beta} + r_2^{1/\beta}  \le 1$, then the open set condition for $\{f_j\}_j$  holds for $V = (0,1)$. 
We remark that each $f_j$ is surjective. 
Hence Assumption \ref{ass:additional-f} holds and Theorem \ref{thm:main} is applicable. 
If $r_1^{1/\beta} + r_2^{1/\beta} > 1$, then, by the uniqueness of the attractor, $K = [0,1]$. 
Hence $\dim_H K = 1/\beta$ and the conclusion of Theorem \ref{thm:main} fails. 
The open set condition also fails.
\end{Exa}

\begin{Exa}
We consider the case where $X = \mathbb{R}^2$ and $d\left((x_1,x_2), (y_1, y_2)\right) = \max\left\{|x_1 - y_1|^{\beta_1},  |x_2 - y_2|^{\beta_2} \right\}$ for some $\beta_1, \beta_2 > 0$. 
The pair $(X,d)$ is a  semimetric space. 
It holds that $B\left((x_1,x_2),r \right) = \left(x_1-r^{1/\beta_1}, x_1+r^{1/\beta_1}\right) \times \left(x_2 - r^{1/\beta_2}, x_2 + r^{1/\beta_2} \right)$ for every $(x_1, x_2) \in X$ and $r > 0$.  
Hence $(X,d)$ is complete and Assumption \ref{ass:two-additional-metric} (ii) holds.

Let $\Phi  (u,v) \coloneqq \max\{1,2^{\beta_1 - 1}, 2^{\beta_2 -1}\} (u+v)$. 
Then this is a triangle function on $(X,d)$. 
Therefore the semimetric space $(X,d)$ is normal and regular and satisfies Assumption \ref{ass:two-additional-metric} (i). 

Let $r_j \in (0,1), \ j \in \{1,2,3,4\}$. 
Let $f_1 (x_1, x_2) \coloneqq (r_{1}^{1/\beta_1} x_1, r_{1}^{1/\beta_2} x_2)$, $f_2 (x_1, x_2) \coloneqq (r_{2}^{1/\beta_1} x_1, r_{2}^{1/\beta_2} x_2 + 1 - r_{2}^{1/\beta_2})$,  $f_3 (x_1, x_2) \coloneqq (r_{3}^{1/\beta_1} x_1 + 1 - r_{3}^{1/\beta_1}, r_{3}^{1/\beta_2} x_2)$, and $f_4 (x_1, x_2) \coloneqq (r_{4}^{1/\beta_1} x_1 + 1 - r_{4}^{1/\beta_1}, r_{4}^{1/\beta_2} x_2 + 1 - r_{4}^{1/\beta_2})$. 
Then, $d(f_j (x_1, x_2), f_j (y_1, y_2)) = r_j d((x_1, x_2), (y_1, y_2))$ for $(x_1, x_2), (y_1, y_2) \in \mathbb{R}^2$ and $ j \in \{1,2,3,4\}$. 
In particular each $f_j$ is a similitude. 
By Theorem \ref{thm:attractor}, there exists a unique attractor $K$ of $\{f_j\}_j$. 

Assume that four open rectangles $(0,r_{1}^{1/\beta_1}) \times (0,r_{1}^{1/\beta_2}), (0,r_{2}^{1/\beta_1}) \times (1-r_{2}^{1/\beta_2},1), (1-r_{3}^{1/\beta_1},1) \times (0,r_{3}^{1/\beta_2}), (1-r_{4}^{1/\beta_1},1) \times (1-r_{4}^{1/\beta_2},1)$ are pairwise disjoint.
Then the open set condition for $\{f_j\}_j$ holds for $V = (0,1) \times (0,1)$. 
We remark that each $f_j$ is surjective. 
Hence Assumption \ref{ass:additional-f} holds and Theorem \ref{thm:main} is applicable. 
We obtain that $\dim_H K = \alpha$ where $\alpha$ is a positive real number satisfying $\sum_{j} r_j^{\alpha} = 1$. 
\end{Exa}

\bibliographystyle{plain}
\bibliography{semimetric-Hutchinson}

\end{document}